\documentclass{amsart}

\usepackage{amssymb,amscd,pb-diagram}
\usepackage[all,arc]{xy}
\usepackage[english]{babel}
\usepackage{tabularx}

\newtheorem{theorem}{Theorem}[section]
\newtheorem*{thm}{Theorem}
\newtheorem{lemma}{Lemma}[section]
\newtheorem{corollary}{Corollary}[section]

\newtheorem{remark}{Remark}[section]
\newtheorem{example}{Example}[section]
\newtheorem*{claim}{Claim}
\numberwithin{equation}{section}

\title{Hartogs type extension theorem for the complement of effective and numerically effective divisors}

\author{S.~V.~Feklistov}
\email{sergeyfe2017@yandex.ru}
\date{}

\begin{document}
\maketitle
\markright{Hartogs extension theorem}
\begin{abstract}
In these notes we generalize the Ohsawa's results on the Hartogs extension phenomenon in the complement of effective divisors in K\"ahler manifolds with semipositive non-flat normal bundle. Namely, we prove that the Hartogs extension phenomenon occurs in the complement of effective and nef divisors with connected supports in K\"ahler manifolds. We use homological algebra methods instead of a construction of the $(n-1)$-convex exhaustion function. Also, the Demailly-Peternell vanishing theorem is a crucial argument for us. Moreover, we obtain geometric characterizations of the Hartogs phenomenon for the complement of basepoint-free divisors.  



\end{abstract}

\section{Introduction}\label{sec1}

Recall that a non-compact complex space $X$ admits the Hartogs phenomenon if for any compact set $K \subset X$ such that $X\setminus K$ is connected, the restriction homomorphism $\Gamma(X, \mathcal{O}_{X}) \to \Gamma (X\setminus K,\mathcal{O}_{X})$ is an isomorphism (here $\mathcal{O}_{X}$ is the sheaf of holomorphic functions). We are interested in the following Ohsawa's result \cite{Ohsawa}. 

\begin{thm}
Let $X$ be a connected compact K\"ahler manifold, let $D$ be an effective divisor on
$X$ , let $\mathcal{O}(D)$ be the line bundle associated to $D$, and let $Z:=Supp(D)$ be the support of $D$. Assume that $\mathcal{O}(D)$ has a fiber metric whose curvature form is semipositive on the
Zariski tangent spaces of $Z$ everywhere and not identically zero at some point of $Z$. Then $X\setminus Z$ admits the Hartogs phenomenon. In particular, $Z$ is connected.
\end{thm}

The first main step of the proof of this theorem is the construction of $(n-1)$-convex exhaustion functions that implies the existence of an $(n-1)$-concave neighborhood system of some connected component of $Z$ (\cite[Proposition 1.2]{Ohsawa}). Secondly, there exist no nonconstant holomorphic functions on any connected neighbourhood of any connected component of $Z$ (\cite[Proposition 1.5]{Ohsawa}). Further, the vanishing theorem (for instance, Demailly-Peternell theorem \cite{Dem1} or $L^{2}$-vanishing theorem) and $\bar{\partial}$-technique implies the assertion of the Theorem.  

Note that in this paper we initially assume that $Z$ is connected and we don't use the argument with a construction of $(n-1)$-convex exhaustion functions. 

In section \ref{sec2} we prove a cohomological criterion of Hartogs phenomenon using classical homological algebra methods (Lemma \ref{mainlemma}). In particular, if $D$ is an effective divisor on a compact complex manifold $X$ with connected support $Z$, then $X\setminus Z$ admits the Hartogs phenomenon if and only if the ideal sheaf $\mathcal{O}(-D)$ of $D$ has no sections near $Z$ (Corollary \ref{crucialcor}).

In section \ref{sec3} we prove the main result (Theorem \ref{mainthm}). Namely, if $D$ is an effective and numerically effective divisor with connected support $Z$ on a compact K\"ahler manifold and $[D]^{2}\neq 0$ in $H^{2,2}(X,\mathbb{R})$ then $X\setminus Z$ admits the Hartogs phenomenon. The Demailly-Peternell vanishing theorem is a crucial for the proof. 

In section \ref{sec4} we consider the case of basepoint-free divisors on compact K\"ahler manifolds. Including Ohsawa's results, we obtain geometrical characterizations of the Hartogs phenomenon (Theorem \ref{basepointthm}). Also, we have a natural question: is it true in more general situations? (see Remarks \ref{rem1} and \ref{rem2}).

In section \ref{sec5} we consider an example in toric geometry context. 

\textit{Acknowledgment}. I am grateful to the referees for their careful reading of the manuscript and useful comments. This work is supported by the Krasnoyarsk Mathematical Center and financed by the Ministry of Science and Higher Education of the Russian Federation (Agreement No. 075-02-2024-1429).

\section{Main lemma and corollaries}\label{sec2}

Let $X$ be a connected compact complex manifold, $\mathcal{O}_{X}$ be the structure sheaf (we often omit the subscript and write $\mathcal{O}$), $Z$ be a connected analytic set of codimension 1, $j\colon Z\hookrightarrow X, i\colon X\setminus Z \hookrightarrow X $ be the canonical embeddings. Let $i^{-1}$ and $j^{-1}$ are the (topological) inverse image functors corresponding to $i$ and $j$. Let $\mathcal{F}$ be a sheaf of $\mathbb{C}$-vector spaces (not necessary coherent) such that $i^{-1}\mathcal{F}\cong i^{-1}\mathcal{O}_{X}=\mathcal{O}_{X\setminus Z}$ (as sheaves of $\mathbb{C}$-vector spaces) and $\dim_{\mathbb{C}} H^{1}(X,\mathcal{F})=m <\infty$. We obtain the following lemma. 

\begin{lemma}\label{mainlemma}
The following assertions are equivalent:
\begin{enumerate}
\item $X\setminus Z$ admits the Hartogs phenomenon;
\item The canonical homomorphism $\Gamma(X,\mathcal{F})\to \Gamma(Z,j^{-1}\mathcal{F})=\varinjlim\limits_{U\supset Z}\Gamma(U,\mathcal{F})$ is epimorphic (where the colimit is taking over open neighbourhoods of $Z$.);
\item $\dim H^{1}_{c}(X\setminus Z,\mathcal{O})\leq m$.
\end{enumerate}
\end{lemma}

\begin{proof}
Recall the following basic facts (see, for instance, \cite{Shapira} or \cite{Manin}): 
\begin{enumerate}
\item For any sheaf $\mathcal{F}$ we have the following short exact sequence $$0\to i_{!}j^{-1}\mathcal{F}\to\mathcal{F}\to j_{*}j^{-1}\mathcal{F}\to 0,$$ where $i_{!}$ is the extension by zero functor, $j_{*}$ is the direct image functor.
\item For any injective sheaf $\mathcal{I}$ we have the following short exact sequence: $$0\to \Gamma_{K}(\mathcal{I})\to\mathcal{I}\to (i_{K})_{*}i_{K}^{-1}\mathcal{I}\to 0,$$ where $i_{K}\colon X\setminus K \hookrightarrow X$ is the canonical open embedding, $\Gamma_{K}(\mathcal{I})$ is the sheaf $U\rightsquigarrow\Gamma_{K}(U,\mathcal{I}):=\{f\in\Gamma(U,\mathcal{I})\mid \text{support of } f \text{ in } K\cap U\}$. 
\item For any sheaf $\mathcal{F}$ we have the canonical isomorphism $$H^{*}_{c}(X,\mathcal{F})\cong\varinjlim\limits_{K} H^{*}_{K}(X,\mathcal{F}),$$ where the colimit is taking over compact sets $K\subset X$.
\end{enumerate} 

It follows the following commutative diagram (here $Y:=X\setminus Z$): \[\tiny
\begin{diagram} 
\node{\Gamma(X,\mathcal{F})} \arrow{e,t}{} \arrow{s,r}{}
\node{\Gamma(Z,j^{-1}\mathcal{F})} \arrow{e,t}{} \arrow{s,r}{}
\node{H^{1}_{c}(Y,\mathcal{O})} \arrow{s,r}{=} \arrow{e,t}{} \node{H^{1}(X,\mathcal{F})}\arrow{s,r}{}
\\
\node{\Gamma(Y,\mathcal{O})} \arrow{e,t}{r} 
\node{\varinjlim\limits_{K\subset Y}\Gamma(Y\setminus K,\mathcal{O})} \arrow{e,t}{c}
\node{H^{1}_{c}(Y,\mathcal{O})} \arrow{e,t}{}\node{H^{1}(Y,\mathcal{O})}  
\end{diagram}\]

Note that since $Z$ and $Y$ are connected, $Y$ admits the Hartogs phenomenon if and only if the canonical homomorphism $r$ is isomorphic. Also, the first square of the diagram is Cartesian because $i^{-1}\mathcal{F}\cong \mathcal{O}_{Y}$.  

The implications $1\Rightarrow 2\Rightarrow 3$ is clear via diagram chasing. 

The implication $3\Rightarrow 1$ follows as in the proof of \cite[corollary 4.3]{AndrHill}. Recall this argument. 

We may assume that $\varinjlim\limits_{K\subset Y}\Gamma(Y\setminus K,\mathcal{O})\neq \mathbb{C}$. Consider an equivalence class $f\in \varinjlim\limits_{K\subset Y}\Gamma(Y\setminus K,\mathcal{O})$ of a non-constant holomorphic function on $Y\setminus K$. Suppose that $c(f)\neq 0$. We may assume that $c(f^{i})$ are non-zero for any $1\leq i\leq m+1$. So, the elements $c(f), c(f^{2}),\cdots, c(f^{m+1})$ are non-zero and linearly dependent. This means that there exists a non-zero polynomial $P\in \mathbb{C}[T]$ of degree $m+1$ such that $c(P(f))=0$. It follows that there is a non-constant holomorphic function $H\in\Gamma(Y,\mathcal{O})$ such that $r(H)=P(f)$. 

Now the elements $ c(f),c(r(H)f),\cdots, c(r(H)^{m}f)$ are non-zero and linearly dependent. Hence there exists a polynomial $P_{1}\in \mathbb{C}[T]$ such that $c(P_{1}(r(H))f)=0$. It follows that there exists a holomorphic function $F\in\Gamma(Y,\mathcal{O})$ such that $r(F)=P_{1}(r(H))f$. Denoting $G=P_{1}(H)$, we obtain $r(F)=r(G)f$. 

Since $r(G^{m+1}H)=r(G^{m+1}P(F/G))$, it follows that $G^{m+1}H=G^{m+1}P(F/G)$. Hence, we obtain $H=P(F/G)$ on $Y\setminus \{G=0\}$. It follows that $F/G\in \Gamma(Y\setminus \{G=0\}, \mathcal{O})$ and locally bounded on $Y\setminus \{G=0\}$. Since $G\neq 0$, the Riemann extension theorem implies that $F/G\in \Gamma(Y,\mathcal{O})$ and $r(F/G)=f$. Therefore, the canonical homomorphism $r$ is isomorphic. 
\end{proof}

Taking $\mathcal{F}=\mathcal{O}_{X}$ we obtain the following corollary. 

\begin{corollary}
$X\setminus Z$ admits the Hartogs phenomenon if and only if $$\Gamma(Z, j^{-1}\mathcal{O})=\mathbb{C}.$$
\end{corollary}

\begin{example}
Let $X$ be a compact smooth complex surface, $Z\subset X$ be a smooth connected complex curve, $N_{Z}$ be the normal bundle of $Z$. The following table is based on the results of \cite{Suzuki}.

\begin{tabular}[t]{|p{4em}|p{11em}|p{5em}|p{10em}|}
\hline
 $c_{1}(N_{Z})$ & the Remmert reduction of $\varepsilon$-neighbourhood of $Z$: & $\Gamma(Z,j^{-1}\mathcal{O})$ & Hartogs phenomenon in $X\setminus Z$\\
\hline
$<0$ & 2-dimensional
normal Stein space & $\neq \mathbb{C}$ & no\\
\hline
$>0$ & point & $=\mathbb{C}$ & yes\\
\hline
$=0$ & point & $=\mathbb{C}$ & yes\\ \cline{2-4} & unit disc in $\mathbb{C}$ & $\neq \mathbb{C}$ &no \\
\hline
\end{tabular}
\end{example}

Now let $D$ be an effective divisor on $X$ such that the support of $D$ is $Z$ (i.e. $Supp(D)=Z$). Taking $\mathcal{F}=\mathcal{O}_{X}(-D)$ (the ideal sheaf of $D$) we obtain the following corollary which is crucial for us. 

\begin{corollary}\label{crucialcor}
$X\setminus Z$ admits the Hartogs phenomenon if and only if there exists an effective divisor $D$ of $X$ whose support coincides with $Z$ such that $$\Gamma(Z, j^{-1}\mathcal{O}(-D))=0.$$
\end{corollary}

\section{Hartogs extension theorem}\label{sec3}

In this section we proved the following Hartogs type extension theorem. 
\begin{theorem}\label{mainthm}
Let $X$ be a compact K\"ahler manifold, $D$ be an effective and numerically effective divisor such that $Z=Supp(D)$ is connected and $[D]^{2}\neq 0$ in $H^{2,2}(X,\mathbb{R})$. Then $X\setminus Z$ admits the Hartogs phenomenon. In particular, we obtain $H^{1}_{c}(X\setminus Z,\mathcal{O})=0$ and $\Gamma (Z,j^{-1}\mathcal{O})=\mathbb{C}$.
\end{theorem}

\begin{proof}
Note that for any $m\in \mathbb{Z}_{>0}$ we obtain that $mD$ is again effective, numerically effective and $[mD]^{2}\neq 0$. By the Demailly-Peternell vanishing theorem \cite{Dem1}, $H^{i}(X,\mathcal{O}(-mD))=0$ for $i=0,1$ and for any $m\in \mathbb{Z}_{>0}$. 

For any $m\in \mathbb{Z}_{>0}$ we have the canonical monomorphism of sheaves $$\mathcal{O}(-mD)\to \mathcal{O}(-D)$$ and the short exact sequence: 
$$0\to i_{!}\mathcal{O}_{X\setminus Z}\to\mathcal{O}(-mD)\to j_{*}j^{-1}\mathcal{O}(-mD)\to 0$$

It follows the following commutative diagram: 
\[\tiny
\begin{diagram} 
\node{\Gamma(X,\mathcal{O}(-mD))=0} \arrow{e,t}{} \arrow{s,r}{}
\node{\Gamma(Z,j^{-1}\mathcal{O}(-mD))} \arrow{e,t}{} \arrow{s,r}{}
\node{H^{1}_{c}(X\setminus Z,\mathcal{O})} \arrow{s,r}{=} \arrow{e,t}{} \node{H^{1}(X,\mathcal{O}(-mD))=0}\arrow{s,r}{}
\\
\node{\Gamma(X,\mathcal{O}(-D))=0} \arrow{e,t}{}
\node{\Gamma(Z,j^{-1}\mathcal{O}(-D))} \arrow{e,t}{}
\node{H^{1}_{c}(X\setminus Z,\mathcal{O})} \arrow{e,t}{} \node{H^{1}(X,\mathcal{O}(-D))=0}
\end{diagram}\]

So, for any $m\in \mathbb{Z}_{>0}$: $\Gamma(Z,j^{-1}\mathcal{O}(-mD))=\Gamma(Z,j^{-1}\mathcal{O}(-D))$. Now, if $f\in \Gamma(Z,j^{-1}\mathcal{O}(-D))$ is non-zero, then $div(f)-mD\geq 0$ for any $m\in \mathbb{Z}_{>0}$ which is impossible. Hence $\Gamma(Z,j^{-1}\mathcal{O}(-D))=0$ and Corollary \ref{crucialcor} implies that $X\setminus Z$ admits the Hartogs phenomenon. In particular, $H^{1}_{c}(X\setminus Z,\mathcal{O})=0$ and $\Gamma (Z,j^{-1}\mathcal{O})=\mathbb{C}$.  
\end{proof}

\section{Remark on basepoint-free divisors}\label{sec4}

Let $X$ be a compact K\"ahler manifold, $D$ be an effective basepoint-free divisor such that $Z=Supp(D)$ is connected. Note that $D$ is nef. Let $\phi_{|D|}\colon X\to \mathbb{P}^{N}$ be the holomorphic map corresponding to the linear system $|D|$. Denote $Y=\phi_{|D|}(X)$. Including Ohsawa's results, we obtain the following geometrical characterizations of the Hartogs phenomenon in this case.

\begin{theorem}\label{basepointthm}
If $X$ is a compact K\"ahler manifold, $D$ is an effective basepoint-free divisor such that $Z=Supp(D)$ is connected, then the following assertions are equivalent: 

\begin{enumerate}
\item $X\setminus Z$ admits the Hartogs phenomenon;
\item $\dim Y>1$;
\item $[D]^{2}\neq 0$ in $H^{2,2}(X,\mathbb{R})$;
\item The normal bundle $\mathcal{O}(D)|_{Z}$ is semipositive and non-flat;
\item $Z$ has an $(n-1)$-concave neighbourhood system.
\end{enumerate}
\end{theorem}

\begin{proof}
$1\Leftrightarrow 2$. The Stein factorization implies a proper surjective holomorphic map $\phi\colon X\to Y$, where $Y$ is a normal projective variety, such that $D=\phi^{*}(H)$ for an ample divisor $H$ on $Y$ and $\phi_{*}\mathcal{O}_{X}=\mathcal{O}_{Y}$. It follows that $\phi\colon X\setminus Z\to Y\setminus supp(H)$ is the Remmert reduction. It is easy to see that $X\setminus Z$ admits the Hartogs phenomenon if and only if $Y\setminus supp(H)$ admits the Hartogs phenomenon. Since $Y\setminus supp(H)$ is normal Stein, then it admits the Hartogs phenomenon if and only if $\dim Y>1$ (see, for instance, \cite{Vassiliadou}). 

$2\Leftrightarrow 3$. We have $D=\phi^{*}_{|D|}H$ for an ample divisor $H$ on $\mathbb{P}^{N}$. Assume $\dim Y>1$. Putting $\omega:=\phi_{|D|}^{*}\omega_{FS}$ (here $\omega_{FS}$ is the Fubini-Study form) we obtain $\omega^{2}$ is nonzero at some point of $Z$. It follows that $[D]^{2}\neq 0$ (via Hodge theory). If $\dim Y=0$ or $1$, then $(\omega_{FS}|_{Y^{reg}})^{2}=0$. 

The implications $2\Rightarrow 4\Rightarrow 5\Rightarrow 1$ follows from Ohsawa's theorem, because we may take $\phi_{|D|}^{*}\omega_{FS}$ as above, where $\omega_{FS}$ is the Fubini-Study form on $\mathbb{P}^{N}$. 
\end{proof}
\begin{remark}\label{rem1}
If $D$ is an effective and numerically effective but not basepoint-free divisor, then the implication $3\Rightarrow 1$ is also true by Theorem \ref{mainthm}. But the implication $1\Rightarrow 3$ is not true in general. Indeed, there is an example of a compact k\"ahler surface $X$ and smooth curve $D$ of $X$ such that $X\setminus D$ is biholomorphic to $(\mathbb{C}^{*})^{2}$, $D$ is not basepoint-free (because $\dim\Gamma (X,\mathcal{O}(mD))=1$ for any $m\in\mathbb{Z}_{>0}$) and $[D]^{2}=0$ in $H^{2,2}(X,\mathbb{R})$ (for more details, see \cite[Example 3.2]{Hartshorne} or \cite[Theorem 2]{Ueda1}). One can generalize this observation to the case where $D$ is a smooth curve of a compact surface which is of class $(\alpha)$ in the classification of $\S 5$ of \cite{Ueda2}. 

See also the paper \cite{Koike} for the relation between the Hartogs phenomenon and the structure of neighbourhoods of $D$ where $D$ is a connected non-singular hypersurface with $[D]^{2}=0$ of a connected compact K\"ahler manifold $X$ \cite[Theorem 1.4]{Koike}. It follows that if $\mathcal{O}(D)$ is semipositive and $[D]^{2}=0$, then $X\setminus D$ may admits or not the Hartogs phenomenon (see \cite[Theorem 1.1]{Koike}). Note that, if $X\setminus D$ does not admit the Hartogs phenomenon in this case, it follows that there exists a neighbourhood $V$ of $D$ such that the line bundle $\mathcal{O}(D)|_{V}$ is unitary
flat by \cite[Theorem 1.4]{Koike} (in particular, $D$ admits a fundamental system of neighbourhoods $\{V_{\varepsilon}\}_{\varepsilon>0}$ such that $\partial V_{\varepsilon}$ are Levi-flat). 

So, the basepoint-free property is important in Theorem \ref{basepointthm}.
\end{remark}
\begin{remark}\label{rem2}
If $D$ is an effective but not necessarily numerically effective divisor, then via the example due to Remark \ref{rem1} the implication $1\Rightarrow 2$ is not true in general because in this example $\dim Y=0$. The implication $ 2\Rightarrow 1$ in this case is true for compact complex manifolds with the finitely generated $\mathbb{C}$-algebra $R(X,D):=\bigoplus\limits_{m\geq 0}\Gamma(X,\mathcal{O}(mD))$. Indeed, consider the meromorphic map $\phi_{|D|}\colon X \dashrightarrow Y$ corresponding to the linear system $|D|$. Replacing $D$ by $mD$ for a suitable integer $m>0$ and using \cite[Proposition 3.1]{Nakayama}, we obtain that there exists a bimeromorphic map $\sigma\colon\widehat{X}\to X$ from a compact normal complex analytic variety $\widehat{X}$, and an effective divisor $F$ of $\widehat{X}$ such that $M:= \sigma^{*}D-F$ is basepoint-free divisor of $\widehat{X}$. The linear system $|M|$ defines a proper surjective holomorphic map $\tau\colon \widehat{X}\to Y$ such that $f\circ \sigma=\tau$ over a Zariski open subset of $\widehat{X}$ and $M=\tau^{*}H$ for an ample divisor $H$ of $Y$. Via the Stein factorization, we may assume that $Y$ is a normal projective variety and $\tau_{*}\mathcal{O}_{\widehat{X}}=\mathcal{O}_{Y}$. Let $Z'$ be the support of $\sigma^{*} D$, $Z''$ be the support of $M$. Note that $\sigma|_{\widehat{X}\setminus Z'}$ is an isomorphism via the construction of $\widehat{X}$ (see proof of \cite[Proposition 3.1]{Nakayama}). We obtain the following: 

\[\small
\begin{diagram} 
\node{X\setminus Z}  \node{\widehat{X}\setminus Z'}  \arrow{e,t}{\text{into open}}
\arrow{w,t}{\text{iso via }\sigma} 
\node{\widehat{X}\setminus Z''} \arrow{e,t}{\text{onto proper}} \node{Y\setminus Supp(H)}
\end{diagram}\]

Since $Y\setminus Supp(H)$ is a normal Stein space and $\dim Y>1$, it admits the Hartogs phenomenon \cite{Vassiliadou}. Since $\tau|_{\widehat{X}\setminus Z''}$ is proper surjective, then $\widehat{X}\setminus Z''$ also admits the Hartogs phenomenon. Since $\widehat{X}\setminus Z'$ is an open complex submanifold of $\widehat{X}\setminus Z''$, it follows that $\widehat{X}\setminus Z'$ also admits the Hartogs phenomenon. Indeed, for the open complex submanifolds we have the following claim:

\begin{claim}
Let $X'$ be a connected compact complex manifold, $W\subset X\subset X'$ are connected open complex submanifolds. If $X$ admits the Hartogs phenomenon, then $W$ is too. 
\end{claim}
\begin{proof}
Let $K\subset W$ be a compact set such that $W\setminus K$ is connected. Since $X, W, W\setminus K$ are connected, it follows that $X\setminus K$ is also connected. Consider the following commutative diagram:

\[\tiny
\begin{diagram} 
\node{\Gamma(X', \mathcal{O})=\mathbb{C}} \arrow{e,t}{r_{2}} \arrow{s,r}{}
\node{\Gamma(X'\setminus K, \mathcal{O})} \arrow{e,t}{} \arrow{s,r}{}
\node{H^{1}_{K}(X',\mathcal{O})} \arrow{e,t}{} \arrow{s,r}{\cong} 
\node{H^{1}(X',\mathcal{O})}
\\
\node{\Gamma(X, \mathcal{O})} \arrow{e,t}{r_{1}} \arrow{s,r}{}
\node{\Gamma(X\setminus K, \mathcal{O})} \arrow{e,t}{} \arrow{s,r}{}
\node{H^{1}_{K}(X,\mathcal{O})} \arrow{s,r}{\cong} 
\\
\node{\Gamma(W, \mathcal{O})} \arrow{e,t}{r_{3}} 
\node{\Gamma(W\setminus K, \mathcal{O})} \arrow{e,t}{} 
\node{H^{1}_{K}(W,\mathcal{O})}
\end{diagram}
\]

Since $r_{1}$ is an isomorphism, then $r_{2}$ is too. Since $\dim_{\mathbb{C}}H^{1}(X',\mathcal{O})<\infty$ it follows that $\dim_{\mathbb{C}} H^{1}_{K}(W,\mathcal{O})<\infty$ which implies that $r_{3}$ is an isomorphism, as in the proof of the implication $3\Rightarrow 1$ of Lemma \ref{mainlemma}.
\end{proof}

Note that, the above arguments work only for the case of finitely generated $\mathbb{C}$-algebras $R(X,D)$; the general case is more complicated and one naturally leads to the Zariski decomposition problem (see \cite[Chapter III]{Nakayama}). 
\end{remark}

\section{Examples}\label{sec5}

In this section we consider some divisors on toric varieties (about toric varieties see, for instance, \cite{Cox}). 

Let $X$ be a projective toric manifold with dense torus $T$. Let $M$ be the lattice of characters of torus $T$ and $M_{\mathbb{R}}=M\otimes_{\mathbb{Z}}\mathbb{R}$, $N:=Hom_{\mathbb{Z}}(M,\mathbb{Z})$ be the lattice of 1-parametric subgroups and $N_{\mathbb{R}}=N\otimes_{\mathbb{Z}}\mathbb{R}$. Let $\Sigma$ be the fan corresponding to $X$. 

Let $f=\sum\limits_{m\in M}a_{m}\chi_{m} \in \mathbb{C}[T]$ be a Laurent polynomial with the Newton polytope $P_{f}:=Conv (m\in M_{\mathbb{R}}\mid a_{m}\neq 0)\subset M_{\mathbb{R}}$. We obtain a rational function $f\in \mathbb{C}(X)$ and let $D=\overline{div(f|_{T})}$ (i.e. the closure in $X$ of zero-divisor of $f|_{T}$) and $D_{\infty}:=D-div(f)$. 

Each ray $\rho_{i}$ of $\Sigma$ gives a minimal ray generator $u_{i}\in N$ and $T$-invariant prime divisor $D_{i}$ of $X_{\Sigma}$.

Let $l_{f}$ be the \textit{support function of the polytope} $P_{f}$, i.e. $l_{f}\colon N_{\mathbb{R}}\to\mathbb{R}, l_{f}(u):=\min\limits_{m\in P_{f}}\langle u,m\rangle.$ Then $D_{\infty}=-\sum l_{f}(u_{i})D_{i}$ \cite{Khov}. 

Note that $D_{\infty}$ is basepoint-free if and only if $D_{\infty}$ is nef (which is equivalent to $D_{\infty}\cdot C\geq 0$ for all $T$-invariant irreducible complete curves $C \subset X$) \cite[Theorem 6.3.12]{Cox}. Also 
there are convex-geometric criteria of nefness in the term of \textit{support function of the divisor} $D_{\infty}$, i.e. $\phi_{D_{\infty}}\colon N_{\mathbb{R}}\to\mathbb{R}$ which is linear on each cone $\sigma\in\Sigma$ and $\phi_{D_{\infty}}(u_{i})=l_{f}(u_{i})$ (see, for instance, \cite[Theorem 6.1.7, Theorem 6.4.9]{Cox}). 

Now we recall the cohomology ring of projective toric manifolds (Jurkiewicz-Danilov theorem, see, for instance, \cite[Theorem 12.4.4]{Cox}). Let $\{\rho_{i}\}_{i=1}^{r}$ be the set of rays of $\Sigma$. In the polynomial ring $\mathbb{Z}[x_{1},\cdots,x_{r}]$, let $\mathcal{I}_{1}$ be the monomial ideal with square-free generators as follows: $$\mathcal{I}_{1}=\langle x_{i_1}\cdots x_{i_{s}}\mid i_{j} \text{ are distinct and } \sum\limits_{j=1}^{s} \rho_{i_{j}} \text{ is not a cone of } \Sigma\rangle.$$
Also, let $\mathcal{I}_{2}$ be the ideal generated by the linear forms $\sum\limits_{i=1}^{r}\langle m, u_{i}\rangle x_{i}$ where $m$ ranges over $M$. 
Then $x_{i}\to [D_{i}]\in H^{2}(X,\mathbb{Z})$ induces a ring isomorphism $$\mathbb{Z}[x_{1},\cdots,x_{r}]/(\mathcal{I}_{1}+\mathcal{I}_{2})\cong H^{*}(X,\mathbb{Z}).$$ 

\begin{corollary}
Let $X, f, D_{\infty}$ are as above and assume that $D_{\infty}$ is effective and nef. Then $X\setminus \overline{\{f=0\}}$ admits the Hartogs phenomenon if and only if $[D_{\infty}]^{2}\neq 0$ in $H^{*}(X,\mathbb{R})$.
\end{corollary}

We apply above to the Hirzebruch surface $H_{r}$ (about Hirzebruch surfaces in toric context, see \cite[Example 3.1.16]{Cox}). Fix a basis $\{e^{1},e^{2}\}$ of $M$ and $\{e_{1},e_{2}\}$ be the dual basis of $N$. In this case, we have torus-invariant divisors $D_{1},D_{2},D_{3},D_{4}$ which correspond to the rays $\rho_{1}=\mathbb{R}_{\geq 0}\langle-e_1+re_{2}\rangle, \rho_{2}=\mathbb{R}_{\geq 0}\langle e_{2}\rangle, \rho_{3}=\mathbb{R}_{\geq 0}\langle e_1\rangle, \rho_{4}=\mathbb{R}_{\geq 0}\langle-e_{2}\rangle$. The cohomology ring (over $\mathbb{R}$) of $X$ is the following: 
$$ H^{*}(X,\mathbb{R})\cong \mathbb{R}[x_{1},x_{2},x_{3},x_{4}]/\langle x_{1}x_{3},x_{2}x_{4},-x_{1}+x_{3},rx_{1}+x_{2}-x_{4}\rangle.$$ 

Note that $[D_{2}]=[D_{4}]-r[D_{3}], [D_{1}]=[D_{3}], [D_{3}]^{2}=0, [D_{3}]\cdot [D_{4}]=1, [D_{4}]^{2}=r$. 

Let $f=\sum\limits_{m\in\mathbb{Z}^{2}}a_{m}z^{m}\in\mathbb{C}[T]=\mathbb{C}[z_{1},z_{1}^{-1},z_{2},z_{2}^{-1}]$ be a Laurent polynomial, $l_{f}$ be the corresponding support function of the Newton polytope $P_{f}$ of $f$. So, we obtain 
$$[D_{\infty}]=-\sum_{i=1}^{4} l_{f}(u_{i})[D_{i}]=(-l_{f}(u_{1})-l_{f}(u_{3})+rl_{f}(u_{2}))[D_{3}]+(-l_{f}(u_{2})-l_{f}(u_{4}))[D_{4}].$$

Computing $[D_{\infty}]\cdot [D_{3}]$ and $[D_{\infty}]\cdot [D_{4}]$ we obtain the following:
\begin{claim}
$D_{\infty}$ is nef and effective if and only if $l_{f}(u_{i})\leq 0$ for all $i=1,2,3,4$ and $l_{f}(u_{1})+l_{f}(u_{3})-rl_{f}(u_{4})\leq 0$
\end{claim}

Computation \begin{multline*}
[D_{\infty}]^{2}= 2(l_{f}(u_{1})+l_{f}(u_{3})-rl_{f}(u_{2}))(l_{f}(u_{2})+l_{f}(u_{4}))+r(l_{f}(u_{2})+l_{f}(u_{4}))^{2}=\\= (2l_{f}(u_{1})+2l_{f}(u_{3}) -rl_{f}(u_{2})+rl_{f}(u_{4}))(l_{f}(u_{2})+l_{f}(u_{4})),
\end{multline*}

gives us the following: 
\begin{claim}
$[D_{\infty}]^{2}\neq 0$ if and only if $l_{f}(u_{2})+l_{f}(u_{4})\neq 0$ and $2(l_{f}(u_{1})+l_{f}(u_{3}))-r(l_{f}(u_{2}))-l_{f}(u_{4}))\neq 0$. 
\end{claim}

\end{document}